\documentclass[11pt]{article}
\usepackage{amsmath,amssymb,amsthm}
\usepackage{enumerate}
\usepackage{hyperref}
\usepackage{cleveref}
\usepackage{mathabx}
\usepackage{blkarray}
\usepackage{tikz}
\usepackage{authblk}
\usepackage{xspace}
\usepackage{thmtools}
\usepackage{thm-restate}
\usepackage{lineno}

\usetikzlibrary{backgrounds}
\usetikzlibrary{shapes}
\usetikzlibrary{positioning}
\hypersetup{
	pdftitle = {Bounds for the Twin-width of Graphs},
	pdfauthor = {Jungho Ahn, Kevin Hendrey, Donggyu Kim, Sang-il Oum}
}
\newcommand\abs[1]{\lvert #1\rvert}
\newtheorem{THM}{Theorem}[section]
\newtheorem{LEM}[THM]{Lemma}
\newtheorem{COR}[THM]{Corollary}
\newtheorem{PROP}[THM]{Proposition}

\theoremstyle{remark}

\theoremstyle{definition}

\newcommand{\tww}{\mathrm{tww}}

\newcommand{\ER}{\text{Erd\H{o}s-R\'{e}nyi}\xspace}
\newcommand{\E}{\mathbb{E}}
\newcommand{\rdeg}{\mathrm{rdeg}}
\newcommand{\DeltaR}{\Delta\!^R}

\usepackage{boxedminipage}

\begin{document}
\date{July 5, 2022}
\title{Bounds for the Twin-width of Graphs}
\author[2,1]{Jungho~Ahn}
\author[1]{Kevin~Hendrey}
\author[2,1]{Donggyu~Kim}
\author[1,2]{Sang-il~Oum}
\affil[1]{Discrete Mathematics Group, Institute for Basic Science (IBS), Daejeon, South~Korea}
\affil[2]{Department of Mathematical Sciences, KAIST, Daejeon,~South~Korea}
\affil[ ]{\small \textit{Email addresses:} \texttt{junghoahn@kaist.ac.kr}, \texttt{kevinhendrey@ibs.re.kr}, \texttt{donggyu@kaist.ac.kr}, \texttt{sangil@ibs.re.kr}}
\footnotetext[1]{All authors are supported by the Institute for Basic Science (IBS-R029-C1).}

\maketitle

\begin{abstract}
	Bonnet, Kim, Thomass\'{e}, and Watrigant~\cite{twin-width1} introduced the \emph{twin-width} of a graph.
	We show that the twin-width of an $n$-vertex graph is less than $(n+\sqrt{n\ln n}+\sqrt{n}+2\ln n)/2$, and the twin-width of an $m$-edge graph 
	for a positive $m$ is less than $\sqrt{3m}+ m^{1/4} \sqrt{\ln m} / (4\cdot 3^{1/4}) +  3m^{1/4} / 2$.
	Conference graphs of order $n$ (when such graphs exist) have twin-width at least $(n-1)/2$, and we show that Paley graphs achieve this lower bound.
	We also show that the twin-width of the Erd\H{o}s-R\'{e}nyi random graph $G(n,p)$ with $1/n\leq p\leq 1/2$ is larger than $2p(1-p)n - (2\sqrt{2}+\varepsilon)\sqrt{p(1-p)n\ln n}$ asymptotically almost surely for any positive $\varepsilon$.
	Lastly, we calculate the twin-width of random graphs $G(n,p)$ with $p\leq c/n$ for a constant $c<1$, determining the thresholds at which the twin-width jumps from $0$ to $1$ and from $1$ to $2$.
\end{abstract}

\section{Introduction}\label{sec:intro}

Bonnet, Kim, Thomass\'{e}, and Watrigant~\cite{twin-width1} introduced a new width parameter, called the \emph{twin-width}, of a graph $G$, denoted by $\tww(G)$.
The twin-width of a graph $G$ captures a way of reducing $G$ into a one-vertex graph by iteratively contracting pairs of vertices with similar neighborhoods (see Section~\ref{subsec:twin} for the formal definiton).
It has been shown that every graph class of bounded twin-width is \emph{small}\footnote[3]{A graph class $\mathcal{C}$ is \emph{small} if $\mathcal{C}$ has at most $n!c^n$ graphs labeled by $[n]$ for a constant $c$.}~\cite{twin-width2}, $\chi$-bounded~\cite{twin-width3}, and admits a linear-time algorithm for first-order model checking~\cite{twin-width1} (if we are given a certificate for an input graph to have small twin-width).
Despite being a relatively new concept, twin-width has already generated a large amount of interest~\cite{twin-width1,twin-width2,twin-width3,twin-width4,twin-width6,GPT2021,BKRTW2021,BH2021,DGJOR2021,ST2021,BNOST2021,SS2021}.

Twin-width can be seen as a generalization of both tree-width and rank-width in the sense that graph classes of bounded tree-width or rank-width also have bounded twin-width~\cite{twin-width1}.
However, there are graph classes that have bounded twin-width and unbounded tree-width or rank-width.
For instance, the $n\times n$ grid has tree-width $n$ and rank-width $n-1$~\cite{squaregrid}, but twin-width at most~$4$~\cite{twin-width1,gridtwinwidth}.
Other interesting classes of bounded twin-width include proper minor-closed classes, map graphs, and $K_t$-free unit $d$-dimensional ball graphs~\cite{twin-width1,twin-width2}.

Given a graph parameter, it is natural to ask for the highest possible value of the parameter.
The maximum tree-width or path-width of an $n$-vertex graph is equal to $n-1$.
By definition, the rank-width of an $n$-vertex graph is bounded above by $\lceil n/3\rceil$, and Lee, Lee, and Oum~\cite{LLO2012} showed that random graphs $G(n,p)$ with a constant $p\in(0,1)$ have rank-width $\lceil n/3\rceil-O(1)$ asymptotically almost surely.
By definition, the linear rank-width of an $n$-vertex graph is bounded above by $\lfloor n/2\rfloor$ and by using the same method of~\cite{LLO2012}, one can show that random graphs $G(n,p)$ with a constant $p\in(0,1)$ have linear rank-width $\lfloor n/2\rfloor-O(1)$ asymptotically almost surely.
Johansson~\cite[Theorem~5.2]{Joh98} showed that the clique-width of an $n$-vertex graph is at most $n-k$ if $n>2^k+k$.
Kneis, M\"{o}lle, Richter, and Rossmanith~\cite{pathbound} showed that the path-width of an $n$-vertex $m$-edge graph is at most $13m/75+O(\log n)$, and used this bound to develop exponential-time algorithms for various problems for sparse graphs.
One can observe from definition that the twin-width of an $n$-vertex graph is at most $n-2$, and currently no nontrivial upper bound for twin-width has appeared in the literature.

Here are our main results.
Firstly, we present an upper bound for the twin-width of an $n$-vertex graph.

\begin{restatable}{THM}{firstmain}\label{main1}
	For a positive integer $n$ and an $n$-vertex graph $G$,
	\begin{linenomath*}\[
		\tww(G)<\frac{1}{2}(n+\sqrt{n\ln n} + \sqrt{n} + 2\ln n).
	\]\end{linenomath*}
\end{restatable}

Secondly, we present an upper bound for the twin-width of a graph $G$ in terms of its number of edges.
This refines the bound in Theorem~\ref{main1} when~$G$ has at most $\abs{V(G)}^2/12$ edges.

\begin{restatable}{THM}{secondmain}\label{main2}
	For a positive integer $m$ and an $m$-edge graph $G$,
	\begin{linenomath*}\[
		\tww(G)<\sqrt{3m}+\frac{m^{1/4} \sqrt{\ln m}}{4\cdot 3^{1/4}} +\frac{3 m^{1/4}}{2}.
	\]\end{linenomath*}
\end{restatable}

Thirdly, we present an asymptotic lower bound for the twin-width of random graphs.

\begin{restatable}{THM}{thirdmain}\label{main3}
	Let $G:=G(n,p)$ be a random graph where $p:=p(n)$ is a function such that $1/n\leq p\leq1/2$ and $\varepsilon$ be a positive real.
	Then
	\begin{linenomath*}\[
		\tww(G) > 2p(1-p)n - (2\sqrt{2}+\varepsilon)\sqrt{p(1-p)n \ln n}
	\]\end{linenomath*}
	asymptotically almost surely.
\end{restatable}

By Theorems~\ref{main1} and~\ref{main3} with $p=1/2$,
\begin{linenomath*}\[
	\lim_{n\rightarrow\infty}\max\left\{\frac{\tww(G)}{\abs{V(G)}}: \abs{V(G)}=n\right\}=\frac{1}{2}.
\]\end{linenomath*}
In terms of the number of edges, we can state the following by Theorems~\ref{main2} and~\ref{main3} with $p=1/3$:
\begin{linenomath*}\[
	1.088<\frac{4\sqrt{6}}{9}\leq\limsup_{m\rightarrow\infty}\max\left\{\frac{\tww(G)}{\sqrt{\abs{E(G)}}}:\abs{E(G)}=m\right\}\leq\sqrt{3}<1.733.
\]\end{linenomath*}
It would be interesting to determine this value exactly.

The best lower bound we have for the maximum twin-width of an $n$-vertex graph comes from conference graphs and is better than the lower bound for random graphs from Theorem~\ref{main3}.
As we will discuss in Section~\ref{sec:paley}, every $n$-vertex conference graph has twin-width at least $(n-1)/2$.
In particular, we prove that Paley graphs achieve this lower bound, which answers an open problem of Schidler and Szeider~\cite{SS2021}.
It is an open problem to determine whether there is an $n$-vertex graph having twin-width at least $n/2$.

\begin{restatable}{THM}{Paley}\label{Paley}
	For each prime power $q$ with $q\equiv 1\pmod{4}$, the Paley graph $P(q)$ has twin-width exactly $(q-1)/2$.
\end{restatable}

We calculate the twin-width of random graphs $G(n,p)$ with $p\leq c/n$ for a constant $c<1$.
This determines the thresholds at which the twin-width jumps from $0$ to $1$ and from $1$ to $2$.

\begin{restatable}{THM}{smallp}\label{smallp}
	For a random graph $G:=G(n,p)$ with a function $p:=p(n)$, the following hold asymptotically almost surely.
	\begin{enumerate}
		\item $\tww(G)=0$ if $p=o(n^{-4/3})$.
		\item $\tww(G)=1$ if $p=\omega(n^{-4/3})$ and $p=o(n^{-7/6})$.
		\item $\tww(G)=2$ if $p=\omega(n^{-7/6})$ and $p \leq c/n$ for a constant $c\in(0,1)$.
	\end{enumerate}
\end{restatable}

We organize this paper as follows.
In Section~\ref{sec:prelim}, we present some terminology from graph theory and probability theory, and formally define twin-width.
In Section~\ref{sec:paley}, we determine the twin-width of Paley graphs.
We prove Theorem~\ref{main1} in Section~\ref{sec:vertex} and Theorem~\ref{main2} in Section~\ref{sec:edge}.
In Section~\ref{sec:asymptotic}, we prove Theorems~\ref{main3} and~\ref{smallp}.

\section{Preliminaries}\label{sec:prelim}

In this paper, all (tri)graphs are simple and finite.
For a set $X$ and a positive integer $t$, $\binom{X}{t}$ is the set of $t$-element subsets of $X$, $[t]:=\{1,\ldots,t\}$, and $[0]:=\emptyset$.
For sets $X$ and $Y$, the \emph{symmetric difference} of $X$ and $Y$, denoted by $X\triangle Y$, is the union of $X\setminus Y$ and $Y\setminus X$.
A \emph{partition} of $X$ is a set $\{X_1,\ldots,X_t\}$ of pairwise disjoint nonempty subsets of~$X$ such that $X=\bigcup_{i=1}^tX_i$.
For a graph $G$ and disjoint vertex sets $X$ and $Y$ of $G$, $X$ is \emph{complete} to $Y$ if every vertex in $X$ is adjacent to all vertices of $Y$, and \emph{anti-complete} if every vertex in $X$ is nonadjacent to all vertices of $Y$.

A \emph{trigraph} is a triple $G=(V,B,R)$ where $B$ and $R$ are disjoint subsets of $\binom{V}{2}$.
Given a trigraph $G=(V,B,R)$, elements in $B$ are \emph{black edges} of $G$ and elements in $R$ are \emph{red edges} of $G$.
We denote the set of black edges of~$G$ by $B(G)$, the set of red edges of~$G$ by $R(G)$, and the set of edges of $G$ by $E(G)$, that is, $E(G)=B(G)\cup R(G)$.
For each vertex $v$ of $G$, the \emph{degree} of~$v$, denoted by $\deg_G(v)$, is the number of edges of $G$ incident with $v$, and the \emph{red-degree} of~$v$, denoted by $\rdeg_G(v)$, is the number of red edges of $G$ incident with $v$.
We denote by $\DeltaR(G)$ the maximum red-degree of a vertex of~$G$.
The \emph{average degree} of $G$ is $2\abs{E(G)}/\abs{V(G)}$.
For a subset $A$ of $V(G)$, $G\setminus A$ is a trigraph obtained from $G$ by removing all vertices in $A$ and all edges incident with vertices in $A$.
Let $\delta_G(A)$ be the set of edges of $G$ having one end in $A$ and the other end in $V(G)\setminus A$.
If $A=\{v\}$ for a vertex~$v$ of~$G$, then we may write $G\setminus v$ and $\delta_G(v)$ instead of $G\setminus\{v\}$ and $\delta_G(\{v\})$, respectively.
We identify a graph $H=(V,E)$ with a trigraph $H=(V,E,\emptyset)$.

\subsection{Twin-width}\label{subsec:twin}

Let $d$ be a nonnegative real.
A trigraph $G = (V,B,R)$ is a \emph{$d$-trigraph} if $\DeltaR(G)\leq d$.
For distinct vertices $u$ and $v$ of~$G$, not necessarily adjacent, $G/\{u,v\}=(V',B',R')$ is a trigraph obtained from $G$ such that $V'=(V\setminus\{u,v\})\cup\{w\}$ for a new vertex~$w$, $G\setminus\{u,v\}=(G/\{u,v\})\setminus w$, and for every vertex $x$ in $V'\setminus\{w\}$, the following hold:
\begin{enumerate}[(i)]
	\item $wx\in B'$ if and only if $\{ux,vx\}\subseteq B$,
	\item $wx\notin B'\cup R'$ if and only if $\{ux,vx\}\cap(B\cup R)=\emptyset$, and
	\item	$wx\in R'$ otherwise.
\end{enumerate}
We say that $G/\{u,v\}$ is obtained from $G$ by \emph{contracting~$u$ and~$v$} into a new vertex~$w$.
We often identify $w$ with either $u$ or $v$.
A \emph{partial contraction sequence} from $G_1$ to $G_t$ is a sequence $G_1,\ldots,G_t$ of trigraphs such that for each $i\in[t-1]$, $G_{i+1}$ is obtained from $G_i$ by contracting two distinct vertices.
A \emph{contraction sequence} of $G$ is a partial contraction sequence from~$G$ to a one-vertex graph.
A \emph{(partial) $d$-contraction sequence} is a (partial) contraction sequence such that every trigraph in the sequence is a $d$-trigraph.
The \emph{twin-width} of $G$, denoted by $\tww(G)$, is the minimum~$d$ such that $G$ admits a $d$-contraction sequence.

A partial contraction sequence $\sigma$ from $G_1$ to $G_t$ defines a function from $V(G_1)$ to $V(G_t)$, which we denote also by $\sigma$, as follows: for a vertex $v$ of~$G_1$, $\sigma(v)$ is the vertex of $G_t$ to which $v$ is contracted.
If $v$ was never contracted, then $\sigma(v)=v$.
For a subset $X$ of $V(G_t)$, we define
\begin{linenomath*}\[
	\sigma^{-1}(X)=\{v\in V(G_1):\sigma(v)\in X\}.
\]\end{linenomath*}
For a vertex $w$ of $G_t$, we may write $\sigma^{-1}(w)$ for $\sigma^{-1}(\{w\})$.
In other words, $\sigma^{-1}(w)$ is the set of vertices of~$G_1$ which are eventually contracted into $w$ in $G_t$.
We have the following two simple lemmas.

\begin{LEM}\label{lem:contraction}
	Let $\sigma$ be a partial contraction sequence from $G$ to $G'$.
	Then for every vertex $v$ of $G'$,
	\begin{linenomath*}\[
		\deg_{G'}(v)\leq\abs{\delta_G(\sigma^{-1}(v))}\leq\sum_{w\in\sigma^{-1}(v)}\deg_G(w).
	\]\end{linenomath*}
\end{LEM}
\begin{proof}
	For each $e:=vw\in\delta_{G'}(v)$, there is at least one edge $f(e)$ of $G$ having one end in $\sigma^{-1}(v)$ and the other end in $\sigma^{-1}(w)$.
	Since $\sigma^{-1}(v)$ and $\sigma^{-1}(w)$ are disjoint, $f(e)$ is in $\delta_G(\sigma^{-1}(v))$.
	For $e':=vw'\in\delta_{G'}(v)$ distinct from $e$, $\sigma^{-1}(w)$ and $\sigma^{-1}(w')$ are disjoint, and therefore $f(e)$ and $f(e')$ are distinct.
	Thus, $f:\delta_{G'}(v)\rightarrow\delta_G(\sigma^{-1}(v))$ is an injection, and this completes the proof.
\end{proof}

The following lemma is immediate from the definition.

\begin{LEM}\label{lem:contraction2}
	Let $\sigma$ be a partial contraction sequence from $G$ to $G'$.
	If $v$ is a vertex of $G'$ with $\abs{\sigma^{-1}(v)}=1$, then $\rdeg_{G'}(v)\leq\rdeg_G(v)+t$, where $t$ is the number of vertices $w$ of $G'$ with $\abs{\sigma^{-1}(w)}>1$.\qed
\end{LEM}

\subsection{Random graphs}

For a positive integer $n$ and $0\leq p\leq1$, the \emph{\ER random graph} $G:=G(n,p)$ is a probability space over the set of graphs with vertex set $[n]$ such that for all $i,j\in[n]$ with $i<j$, the event that $i$ is adjacent to $j$ occurs with probability $p$, and these events are mutually independent.
For a function $p : \mathbb{N} \to [0,1]$, we say that the random graph $G(n,p(n))$ has a graph property $A$ \emph{asymptotically almost surely} if $\Pr[G(n,p(n))\models A]\rightarrow1$ as $n\rightarrow\infty$.

We will use the following inequality to investigate the asymptotic behavior of the twin-width of random graphs.

\begin{LEM}[Chernoff bound; see~{\cite[Theorem~2.1]{randomgraphs}}]\label{lem:Chernoff}
	Let $X_1,\ldots,X_n$ be mutually independent random variables such that for each $i\in[n]$, $\Pr[X_i=1]=p$ and $\Pr[X_i=0]=1-p$.
	Let $X:=\sum_{i=1}^nX_i$.
	Then
	\begin{linenomath*}\[
		\Pr[X-\E[X]\leq-t]\leq\exp\left(-\frac{t^2}{2np}\right)
	\]\end{linenomath*}
	for any $t\geq0$.
\end{LEM}

\subsection{Lattice paths}

For points $\mathbf{x}$ and $\mathbf{y}$ in $\mathbb{Z}^2$, a \emph{lattice path} from $\mathbf{x}$ to $\mathbf{y}$ is a sequence $\mathbf{p}_0,\ldots,\mathbf{p}_k$ of points in $\mathbb{Z}^2$ such that $\mathbf{p}_0 = \mathbf{x}$, $\mathbf{p}_k = \mathbf{y}$, and $\mathbf{p}_i-\mathbf{p}_{i-1}\in\{(0,\pm1),(\pm1,0)\}$ for each $i\in[k]$.
A lattice path $\mathbf{p}_0,\ldots,\mathbf{p}_k$ is \emph{North-East} if $\mathbf{p}_i-\mathbf{p}_{i-1}\in\{(0,1),(1,0)\}$ for each $i\in[k]$.
We say that a lattice path $\mathbf{p}_0,\ldots,\mathbf{p}_k$ \emph{intersects} a line $y=ax+b$ if $\{\mathbf{p}_0,\ldots,\mathbf{p}_k\}$ intersects both $\{(x,y):y\geq ax+b\}$ and $\{(x,y):y\leq ax+b\}$.
Observe that if $\mathbf{p}_i$ is on the line $y=ax+b$ for some $i\in[k]\cup\{0\}$, then the lattice path $\mathbf{p}_0,\ldots,\mathbf{p}_k$ intersects the line, but the converse is not necessarily true if $(a,b)\notin\{-1,0,1\}\times\mathbb{Z}$.

\begin{PROP}[See~{\cite[Theorem~10.3.1]{Bona2015}}]\label{prop:latticepath}
	Let $a$, $b$, and $t$ be nonnegative integers with $a+t>b$.
	Then $\binom{a+b}{a+t}$ is the number of North-East lattice paths from $(0,0)$ to $(a,b)$ intersecting the line $y=x+t$.\footnote[4]{We remind the readers that for integers $n$ and $m$ with $m>n\geq 0$, $\binom{n}{m}$ is defined as $0$.}
\end{PROP}

The following corollary will help us to prove Theorem~\ref{main1}.

\begin{COR}\label{cor:latticepath2}
	Let $a$ and $b$ be nonnegative integers and $t$ be a positive integer.
	Let $A$ be the event that a North-East lattice path from $(0,0)$ to $(a,b)$ chosen uniformly at random intersects the line $y=x+t$.
	Then $\Pr[A]\leq(b/(a+t))^t$.
\end{COR}
\begin{proof}
	We may assume that $0<t\leq b$, because otherwise $\Pr[A]=0$.
	We may assume that $a+t>b$, because otherwise $(b/(a+t))^t\geq1$.
	Since $\binom{a+b}{a}$ is the number of North-East lattice paths from $(0,0)$ to $(a,b)$, by Proposition~\ref{prop:latticepath},
	\begin{linenomath*}\[
		\Pr[A]
		=\frac{\binom{a+b}{a+t}}{\binom{a+b}{a}}
		=\frac{(a+b)!}{(a+t)!\cdot(b-t)!}\cdot\frac{a!\cdot b!}{(a+b)!}
		=\prod_{i=0}^{t-1}\frac{b-i}{a+t-i}
		\leq\left(\frac{b}{a+t}\right)^t,
	\]\end{linenomath*}
	where the last inequality holds because $a+t>b$.
\end{proof}

\section{Twin-width of Paley graphs}\label{sec:paley}

The following lemma follows immediately from the definition of the twin-width.

\begin{LEM}\label{lem:firststep}
	Every $n$-vertex graph $G$ with $n\geq2$ has twin-width at least $\min\{\abs{(N_G(v)\triangle N_G(w))\setminus\{v,w\}}:v,w\in V(G),\ v\neq w\}$.\qed
\end{LEM}

We provide an infinite class of graphs such that every $n$-vertex graph in the class has twin-width at least $(n-1)/2$.
A \emph{conference graph} is a graph~$G$ on $n$ vertices for a positive integer $n$ with $n\equiv 1\pmod{4}$ such that $G$ is $((n-1)/2)$-regular and for each pair $\{v,w\}$ of distinct vertices of $G$, $\abs{N_G(v)\cap N_G(w)}$ is $(n-5)/4$ if $v$ and $w$ are adjacent, and otherwise $(n-1)/4$.
Observe that in an $n$-vertex conference graph $G$, for distinct vertices $v$ and $w$ of $G$, $\abs{(N_G(v)\triangle N_G(w))\setminus\{v,w\}}=(n-1)/2$.
Thus, by Lemma~\ref{lem:firststep}, every $n$-vertex conference graph has twin-width at least $(n-1)/2$.

For each prime power $q$ with $q\equiv 1\pmod{4}$, the \emph{Paley graph} $P(q)$ is a graph on the field $\mathbb{F}_q$ of order $q$ such that vertices $a$ and $b$ are adjacent in $P(q)$ if and only if $a-b = c^2$ for some $c\in \mathbb{F}_q \setminus \{0\}$.
We remark that $-1$ is a square in $\mathbb{F}_q$ for a prime power $q$ with $q\equiv 1\pmod{4}$, so that the Paley graph $P(q)$ is well defined as an undirected graph.
Erd\H{o}s and R\'{e}nyi~\cite{ER1963} showed that every Paley graph is a conference graph.

We now prove Theorem~\ref{Paley}.

\Paley*

\begin{proof}
	Since every Paley graph is a conference graph, by Lemma~\ref{lem:firststep}, it suffices to show that $\tww(P(q)) \leq (q-1)/2$.
	Let $u_1,\ldots,u_{(q-1)/2}$ be elements in $\mathbb{F}_q\setminus\{0\}$ such that for all distinct $i,j\in[(q-1)/2]$, $u_i\neq u_j$ and $u_i\neq-u_j$.
	Let $G_0:=P(q)$ and for each $k\in[(q-1)/2]$, let $G_k:=G_{k-1}/\{u_k,-u_k\}$.
	We will identify the vertex created by contracting $u_k$ and $-u_k$ with $u_k$.
	Since $\abs{V(G_{(q-1)/2})}=(q+1)/2$, it is enough to show that $\DeltaR(G_k)\leq(q-1)/2$ for each $k\in[(q-1)/2]$.

	By Lemma~\ref{lem:contraction2}, for each $u\in V(G_k)\setminus\{u_1,\ldots,u_k\}$, we have $\rdeg_{G_k}(u)\leq k\leq(q-1)/2$.
	Let $i\in[k]$ and $Q_i:=G_0/\{u_i,-u_i\}$.
	We will again identify the new vertex created by contracting $u_i$ and $-u_i$ with $u_i$.
	Since $G_0$ is a conference graph, we have $\rdeg_{Q_i}(u_i)=(q-1)/2$.
	We show that for each $j\in[k]\setminus\{i\}$, one of the following holds.
	\begin{enumerate}[(i)]
		\item $Q_i$ has black edges from $u_i$ to both $u_j$ and $-u_j$.
		\item $u_i$ is nonadjacent to both $u_j$ and $-u_j$ in $Q_i$.
		\item $Q_i$ has at least one red edge from $u_i$ to $u_j$ or $-u_j$.
	\end{enumerate}
	Suppose that $Q_i$ has no red edges from $u_i$ to $u_j$ or $-u_j$.
	Then each of $u_j$ and $-u_j$ is either complete or anti-complete to $\{u_i,-u_i\}$ in $G_0$.
	Since $u_i-u_j=(-u_j)-(-u_i)$, by the definition of the Paley graph, $u_i$ and $u_j$ are adjacent in $G_0$ if and only if $-u_i$ and $-u_j$ are adjacent in $G_0$.
	Therefore, $\{u_j,-u_j\}$ is either complete or anti-complete to $\{u_i,-u_i\}$ in $G_0$.
	This implies that (i) or (ii) holds.
	There is a partial contraction sequence from $Q_i$ to $G_k$ by contracting pairs of the form $\{u_j,-u_j\}$ for all $j\in[k]\setminus\{i\}$.
	We observe from (i), (ii), and (iii) that in this sequence, the red-degree of $u_i$ does not increase.
	Therefore, $\rdeg_{G_k}(u_i)\leq\rdeg_{Q_i}(u_i)$, so that $\DeltaR(G_k)\leq(q-1)/2$, and this completes the proof.
\end{proof}

\section{Twin-width versus the number of vertices}\label{sec:vertex}

For an $n$-vertex graph, the first step of our strategy to prove Theorem~\ref{main1} will be to use the following proposition to find more than $(n-\sqrt{n})/2$ disjoint pairs of vertices such that contracting any one of these pairs creates a trigraph whose maximum red-degree is at most $(n+\sqrt{n}-1)/2$.
Later, we will show that contracting these pairs in a random order will create no intermediate trigraphs having too large red-degree with positive probability.

\begin{PROP}\label{prop:L}
	For an $n$-vertex graph $G$, every subset $L$ of $V(G)$ with $\abs{L}\geq2$ contains distinct vertices $u$ and $v$ such that
	\begin{linenomath*}\[
		\abs{(N_G(u)\triangle N_G(v))\setminus\{u,v\}}\leq\frac{n}{2}+\frac{n-1}{2\abs{L}-2}-1.
	\]\end{linenomath*}
\end{PROP}
\begin{proof}
	Let $\ell:=\abs{L}$.
	The number of elements $\{u,v\}$ in $\binom{L}{2}$ with $w\in(N_G(u)\triangle N_G(v))\setminus\{u,v\}$ is $\abs{L\cap N_G(w)}\cdot(\abs{L\setminus N_G(w)}-1)$ if $w\in L$, and is $\abs{L\cap N_G(w)}\cdot\abs{L\setminus N_G(w)}$ otherwise.
	We choose a pair $\{u,v\}$ from $\binom{L}{2}$ uniformly at random.
	Then
	\begin{linenomath*}\postdisplaypenalty=0\begin{align*}
		&\E[\abs{(N_G(u)\triangle N_G(v))\setminus\{u,v\}}]\\
		&=\sum_{w\in L} \frac{\abs{L\cap N_G(w)}\cdot(\abs{L \setminus N_G(w)}-1)}{\binom{\ell}{2}}
		+\sum_{w\in V(G)\setminus L}\frac{\abs{L\cap N_G(w)}\cdot\abs{L\setminus N_G(w)}}{\binom{\ell}{2}}\\
		&\leq\ell\cdot\frac{\left(\frac{\ell-1}{2}\right)^2}{\binom{\ell}{2}}+(n-\ell)\cdot\frac{\left(\frac{\ell}{2}\right)^2}{\binom{\ell}{2}}=\frac{n}{2}+\frac{n-1}{2\ell-2}-1.
	\end{align*}\end{linenomath*}
	Thus, $\binom{L}{2}$ contains an element $\{u',v'\}$ with $\abs{(N_G(u')\triangle N_G(v'))\setminus\{u',v'\}}\leq n/2+(n-1)/(2\ell-2)-1$.
\end{proof}

We remark that applying Proposition~\ref{prop:L} with $L=V(G)$ yields the following corollary, which indicates that the lower bound in Lemma~\ref{lem:firststep} is at most $(n-1)/2$ for an $n$-vertex graph with $n\geq2$.

\begin{COR}
	Every $n$-vertex graph $G$ with $n\geq2$ has distinct vertices $v$ and $w$ such that $\abs{(N_G(v)\triangle N_G(w))\setminus\{v,w\}}\leq(n-1)/2$.\qed
\end{COR}

For any ordering of the pairs obtained by Proposition~\ref{prop:L}, the following simple lemma presents an upper bound for the twin-width.

\begin{LEM}\label{lem:orderpairs}
	Let $G$ be an $n$-vertex graph for $n\geq2$ and $p_1,\ldots,p_s$ be pairwise disjoint elements in $\binom{V(G)}{2}$.
	For each $i\in[s]$, let~$v_i$ be the new vertex of $G/p_i$ and $\mathbf{x}_i:[s]\rightarrow\{-1,0,1\}$ be the function such that
	\begin{linenomath*}\[
		\mathbf{x}_i(j) =
		\begin{cases}
			\rdeg_{G/p_i/p_j}(v_i)-\rdeg_{G/p_i}(v_i) & \text{if $i\neq j$,} \\
			0 & \text{otherwise.}
		\end{cases}
	\]\end{linenomath*}
	Then the twin-width of $G$ is at most
	\begin{linenomath*}\[
		\max \left( \{n-s-1\} \cup \Bigl\{ \rdeg_{G/p_i}(v_i)+\sum_{k=1}^j\mathbf{x}_i(k): 1\leq i\leq j\leq s-1\Bigr\} \right).
	\]\end{linenomath*}
\end{LEM}
\begin{proof}
	Let $G_0:=G$ and $G_i:=G_{i-1}/p_i$ for each $i\in [s]$.
	We also denote the new vertex of $G_i$ by $v_i$.
	Since $G_0,\ldots,G_s$ is a partial contraction sequence from $G$ to $G_s$, $\tww(G) \leq \max\{\DeltaR(G_0),\ldots,\DeltaR(G_{s-1}),\tww(G_s)\}$.
	Since $G_s$ has $n-s$ vertices, $\tww(G_s)<\abs{V(G_s)}=n-s$.
	Let $j\in[s-1]$.
	Since $\DeltaR(G)=0$, by Lemma~\ref{lem:contraction2}, for each vertex $v$ in $V(G_j)\setminus \{v_1,\dots,v_j\}$, $\rdeg_{G_j}(v)\leq j$.
	For each $i\in[j]$, by the definition of $\mathbf{x}_i$, $\rdeg_{G_j}(v_i)=\rdeg_{G/p_i}(v_i)+\sum_{k=1}^j\mathbf{x}_i(k)$.
	Therefore,
	\begin{linenomath*}\[
		\DeltaR(G_j)\leq\max \left( \{j\} \cup \Bigl\{ \rdeg_{G/p_i}(v_i)+\sum_{k=1}^j\mathbf{x}_i(k): i\in[j]\Bigr\} \right),
	\]\end{linenomath*}
	and this completes the proof, because $j\leq s-1 \leq n-s-1$.
\end{proof}

We now prove Theorem~\ref{main1}.

\firstmain*

\begin{proof}
	We may assume that $n \geq 3$, since every graph on at most two vertices has twin-width $0$.
	Let $M$ be a maximal set of pairwise disjoint elements in $\binom{V(G)}{2}$ such that
	\begin{linenomath*}\[
		\abs{(N_G(u)\triangle N_G(v))\setminus\{u,v\}}\leq\frac{n}{2} +\frac{n-1}{2(\sqrt{n}-1)} - 1=\frac{n+\sqrt{n}-1}{2}
	\]\end{linenomath*}
	and each $\{u,v\}\in M$.
	By Proposition~\ref{prop:L} with $L=V(G)\setminus\bigcup_{Q\in M}Q$ and the maximality of $M$, we have $\abs{L}<\sqrt{n}$, and therefore $\abs{M}=(n-\abs{L})/2>(n-\sqrt{n})/2$.

	We show that there is an ordering of the elements in $M$ such that if we contract the elements in~$M$ in this order, then the maximum red-degree is not too large.
	Let $s:=\abs{M}$ and $p_1,\ldots,p_s$ be the elements in $M$ and for each $i\in[s]$, $v_i$ be the new vertex of $G/p_i$ and $\mathbf{x}_i:[s]\rightarrow\{-1,0,1\}$ be the function such that
	\begin{linenomath*}\[
		\mathbf{x}_i(j) =
		\begin{cases}
			\rdeg_{G/p_i/p_j}(v_i)-\rdeg_{G/p_i}(v_i) & \text{if $i\neq j$,} \\
			0 & \text{otherwise.}
		\end{cases}
	\]\end{linenomath*}
	Let $d:=(n+\sqrt{n\ln n} + \sqrt{n} + 2\ln n)/2$.
	Let $\rho$ be a permutation of $[s]$ chosen uniformly at random.
	For each~$i\in[s]$, let~$A_i$ be the event that $i\neq\rho(s)$ and
	\begin{linenomath*}\[
		\rdeg_{G/p_i}(v_i)+\max\left\{\sum_{k=1}^j\mathbf{x}_i(\rho(k)):\rho^{-1}(i)\leq j\leq s-1\right\}
		\geq d.
	\]\end{linenomath*}
	Note that if $i\neq\rho(s)$, then there is some $j$ with $\rho^{-1}(i)\leq j\leq s-1$, so $A_i$ is well defined.
	
	We first show that $\Pr[A_i]<1/s$ for each $i\in[s]$.
	We fix an element~$i$ in~$[s]$.
	Let $\alpha := |\{j\in[s]: \mathbf{x}_i(j)=-1\}|$, $\beta := |\{j\in[s] : \mathbf{x}_i(j)=1\}|$, and $\gamma:=s-\alpha-\beta$.
	We can construct a North-East lattice path from $(0,0)$ to $(\alpha,\beta)$ as follows.
	For each $j\in[s]$ with $\mathbf{x}_i(\rho(j))\neq0$, let $a_j:=\abs{\{j'\in[j]:\mathbf{x}_i(\rho(j'))\neq0\}}$, and
	\[
		\mathbf{p}_0:=(0,0),\quad
		\mathbf{p}_{a_j}:=
		\begin{cases}
			\mathbf{p}_{a_j-1}+(0,1) & \text{if $\mathbf{x}_i(\rho(j))=1$},\\
			\mathbf{p}_{a_j-1}+(1,0) & \text{if $\mathbf{x}_i(\rho(j))=-1$}.
		\end{cases}
	\]
	Then $\mathbf{p}_0,\ldots,\mathbf{p}_{\alpha+\beta}$ is a North-East lattice path from $(0,0)$ to $(\alpha,\beta)$ chosen uniformly at random.

	Since
	\begin{linenomath*}\[
		\max\left\{\sum_{k=1}^j\mathbf{x}_i(\rho(k)):\rho^{-1}(i)\leq j\leq s-1\right\}\leq\beta\leq s-1,
	\]\end{linenomath*}
	we may assume that $\rdeg_{G/p_i}(v_i)\geq d-s+1$.
	Let $B_i$ be the event that
	\begin{linenomath*}\[
		\rdeg_{G/p_i}(v_i)+\max\left\{\sum_{k=1}^j\mathbf{x}_i(\rho(k)):1\leq j\leq s\right\}\geq d.
	\]\end{linenomath*}
	We observe that $A_i\subseteq B_i$ and that $\Pr[B_i]$ is equal to the probability that a North-East lattice path from $(0,0)$ to $(\alpha,\beta)$ chosen uniformly at random intersects the line $y=x+t$ where
	\begin{linenomath*}\[
		t:=-\rdeg_{G/p_i}(v_i)+\lceil d\rceil.
	\]\end{linenomath*}
	Since $\rdeg_{G/p_i}(v_i) \leq (n+\sqrt{n}-1)/2$, we have $t \geq d-(n+\sqrt{n}-1)/2=(\sqrt{n\ln n}+1)/2+\ln n>0$.
	By Corollary~\ref{cor:latticepath2}, $\Pr[B_i]\leq(\beta/(\alpha+t))^t$.
	Thus, it suffices to show that $(\beta/(\alpha+t))^t<1/s$.
	For $i,j\in[s]$, we observe that $\mathbf{x}_i(j)=-1$ if and only if $G/p_i$ has two red edges from $v_i$ to $p_j$.
	In addition, if $i=j$ or $G/p_i$ has exactly one red edge from $v_i$ to $p_j$, then $\mathbf{x}_i(j)=0$.
	Therefore, the number of vertices in $\bigcup_{Q\in M\setminus\{p_i\}}Q$ which are adjacent to $v_i$ by red edges in $G/p_i$ is at most $2\alpha+\gamma-1$.
	Thus,
	\begin{linenomath*}\postdisplaypenalty=0\begin{align*}
		d
		\leq t+\rdeg_{G/p_i}(v_i)
		&\leq t+2\alpha+\gamma-1+\left|V(G)\setminus\bigcup_{Q\in M}Q\right|\\
		&=t+\alpha-\beta-1+n-s,
	\end{align*}\end{linenomath*}
	so that $\beta\leq t+\alpha-d-1+n-s$.
	Then since $1-x\leq e^{-x}$ for all real $x$,
	\begin{linenomath*}\[
		\left(\frac{\beta}{\alpha+t}\right)^t
		\leq\left(1-\frac{d+s+1-n}{\alpha+t}\right)^t
		\leq\exp\left(-\frac{t(d+s+1-n)}{\alpha+t}\right).
	\]\end{linenomath*}
	Since $1/n<1/s$, it suffices to show that $t(d+s+1-n)\geq(\alpha+t)\ln n$, that is, $t(d+s+1-n-\ln n)\geq\alpha\ln n$.
	Since $2\alpha\leq\rdeg_{G/p_i}(v_i) \leq (n+\sqrt{n}-1)/2$, we have $\alpha\ln n \leq (n\ln n+\sqrt{n}\ln n - \ln n)/4$.
	Note that $t \geq (\sqrt{n\ln n}+1)/2+\ln n$.
	Since $s>(n-\sqrt{n})/2$, we have $d+s>\sqrt{n\ln n}/2+n+\ln n$.
	Thus,
	\begin{linenomath*}\postdisplaypenalty=0\begin{align*}
		&t(d+s+1-n-\ln n)-\alpha\ln n\\
		&>\left(\frac{\sqrt{n\ln n} + 1}{2}+\ln n\right) \frac{\sqrt{n\ln n}+2}{2}-\frac{n\ln n+\sqrt{n}\ln n - \ln n}{4}\\
		&>\left(\frac{\sqrt{n\ln n} }{2}+\ln n\right) \frac{\sqrt{n\ln n}}{2}-\frac{n\ln n+\sqrt{n}\ln n}{4}\\
		&=\frac{\sqrt{n} (\ln n)^{3/2}}{2} - \frac{\sqrt{n} \ln n}{4} > 0,
	\end{align*}\end{linenomath*}
	and therefore $\Pr[A_i]\leq\Pr[B_i]<1/s$.

	By the union bound, $\Pr[\bigcap_{i=1}^{s}\overline{A_i}]>0$, that is, there exists a permutation~$\rho_0$ of $[s]$ such that $\rdeg_{G/{p_i}}(v_i) + \max\{\sum_{k=1}^j\mathbf{x}_i(\rho_0(k)):\rho_0^{-1}(i)\leq j\leq s-1\}<d$ for each $i\in[s] \setminus \{\rho_0(s)\}$.
	Note that $n-s-1<(n+\sqrt{n}-2)/2<d$.
	Therefore, by Lemma~\ref{lem:orderpairs}, the twin-width of~$G$ is less than $d$.
\end{proof}

\section{Twin-width versus the number of edges}\label{sec:edge}

We now prove Theorem~\ref{main2} which improves the bound in Theorem~\ref{main1} for any $n$-vertex graph having at most $n^2/12$ edges.
Let us restate the theorem.

\secondmain*

We first briefly sketch the proof.
We will construct a contraction sequence for $G$ in three phases.
In the first phase, for some $k\approx\sqrt{4m/3}$, we find at most $k$ disjoint sets of vertices, covering all but at most $k$ vertices of~$G$, such that the sum of the degrees of the vertices within each of the sets is at most $2m/k$.
We begin the contraction sequence by contracting each set into a new vertex.
We can do this in such a way that the red-degree will not exceed the maximum of $k$ and $2m/k$.

In the second phase, we apply Theorem~\ref{main1} to obtain a contraction sequence for the graph induced on the set of vertices not covered by the disjoint sets in the first phase.
Note that throughout this contraction sequence, the red-degree of a vertex created in the first phase is still at most the sum of the degrees of the vertices in the corresponding set.

In the final phase, a small number of vertices remain, so taking an arbitrary contraction sequence of the resulting trigraph suffices.

We will need the following lemma for the first phase.
While it is equivalent to restrict $k$ to be an integer in the following lemma, we use the current form for convenience.

\begin{LEM}\label{lem:partition}
	Let $a_1,\ldots,a_n$ be a sequence of nonnegative reals and $d:=\sum_{i=1}^n a_i/n$.
	For a positive real $k$, there exist a subset $R$ of $[n]$ with $\abs{R}<k$ and a partition $\{B_1,\ldots,B_{k'}\}$ of $[n]\setminus R$ such that $k'\leq\lceil k\rceil$ and $\sum_{j\in B_i} a_j \leq dn/k$ for each $i\in[k']$.
\end{LEM}
\begin{proof}
	We proceed by induction on $\lceil k \rceil$.
	We may assume that $k>1$, because otherwise we can take $B_1=[n]$.
	We may assume that $k\leq n$, because otherwise we can take $R=[n]$.
	Let $B$ be a maximal subset of $[n]$ such that $\sum_{j\in B}a_j\leq dn/k$.
	Since $k\leq n$, there is an integer $j$ in $[n]$ with $a_j\leq dn/k$, because otherwise $\sum_{j=1}^na_j>dn^2/k\geq dn$.
	Thus, $B$ exists and is nonempty.
	In addition, since $k>1$, $B$ is a proper subset of~$[n]$.
	Let $r$ be an element in~$[n]\setminus B$.
	Without loss of generality, we may assume that $[n]\setminus(B\cup\{r\})$ is equal to $[n']$ for some positive integer $n'$.
	Let $d' := \sum_{j=1}^{n'} a_j / n'$.
	By the maximality of~$B$, $dn-d'n'= \sum_{j \in B \cup \{r\}} a_j>dn/k$, and therefore $d'n'/(k-1)<dn/k$.
	By the inductive hypothesis, there exist a subset $R$ of $[n']$ with $\abs{R}<k-1$ and a partition $\{B_1,\ldots,B_t\}$ of $[n']\setminus R$ such that $t\leq\lceil k-1\rceil$ and $\sum_{j\in B_i}a_j\leq d'n'/(k-1)<dn/k$ for each $i\in[t]$.
	Then the subset $R\cup\{r\}$ and the partition $\{B_1,\ldots,B_t,B\}$ of $[n]\setminus(R\cup\{r\})$ satisfy the required conditions.
\end{proof}

We now prove Theorem~\ref{main2}.

\begin{proof}[Proof of Theorem~\ref{main2}]
	It is easy to check that if $m\leq 3$, then $\tww(G) \leq 1$.
	Thus, we may assume that $m\geq4$.
	Let $n:=|V(G)|$ and $q:=\sqrt{4m/3}$.
	Note that $2<q<n$, because $4\leq m<n^2/2$.
	Let $d$ be the average degree of $G$,
	\begin{linenomath*}\postdisplaypenalty=0\begin{align*}
		k&:=q-\frac{\sqrt{q\ln q}+\sqrt{q}}{6}-\frac{\ln q}{9},\\
		\alpha&:=\frac{3}{2}\left(q+\frac{\sqrt{q\ln q}+\sqrt{q}}{6}+\frac{5\ln q}{9}\right).
	\end{align*}\end{linenomath*}
	Since $q>\ln q$, we have $q > \sqrt{q\ln q}>\ln q$, and therefore $1<q/2<k<q$.
	By Lemma~\ref{lem:partition}, there exist a subset $R$ of $V(G)$ with $r:=\abs{R}<k$ and a partition $\{B_1,\ldots,B_{k'}\}$ of $V(G)\setminus R$ such that $k'\leq\lceil k\rceil$ and $\sum_{v\in B_i} \deg_G(v) \leq dn/k$ for each $i\in[k']$.
	We may assume that $R$ is nonempty by moving one vertex in $B_i$ into $R$ for some $i\in[k']$ if necessary.
	Since $k<q<n$, $R$ is a proper subset of $V(G)$ and $k'\geq1$.
	For each $i\in[k']$, let $\ell_i:=\abs{B_i}$ and $v_{i,1},\ldots,v_{i,\ell_i}$ be the vertices in~$B_i$.

	Now, we construct an $\lceil\alpha\rceil$-contraction sequence of $G$.
	Let $G_1:=G$ and for $i\in[k']$ and $j\in[\ell_i]\setminus\{1\}$,
	\begin{linenomath*}\[
		G_{j+\sum_{h=1}^{i-1}(\ell_h-1)}:=G_{j-1+\sum_{h=1}^{i-1}(\ell_h-1)}/\{v_{i,1},v_{i,j}\}.
	\]\end{linenomath*}
	Here, the new vertex created by contracting $v_{i,1}$ and $v_{i,j}$ will be called $v_{i,1}$ for convenience.
	Let $t:=\sum_{h=1}^{k'}(\ell_h-1)+1=n-k'-r+1$.
	We observe that $G[R]=G_t[R]$.
	By Theorem~\ref{main1}, $G[R]$ admits an $((r+\sqrt{r\ln r}+\sqrt{r}+2\ln r)/2)$-contraction sequence $H_1,\ldots,H_r$.
	For each $i\in[r-1]$, let $x_i$ and $y_i$ be the vertices in $H_i$ such that $H_{i+1}=H_i/\{x_i,y_i\}$ and define
	\begin{linenomath*}\[
		G_{t+i}:=G_{t+i-1}/\{x_i,y_i\}.
	\]\end{linenomath*}
	Finally, we take an arbitrary contraction sequence $G_{t+r-1},\ldots,G_n$ of $G_{t+r-1}$.
	We show that $\DeltaR(G_s) \leq \lceil \alpha \rceil$ for each $s\in[n]$.
	For each $i\in[n]$, let $\sigma_i$ be the partial contraction sequence $G_1,\ldots,G_i$.
	Let $s\in[n]$ and $v$ be a vertex of~$G_s$.
	
	Firstly, suppose that $s\leq t$.
	If $\abs{\sigma_s^{-1}(v)}=1$, then by Lemma~\ref{lem:contraction2}, $\rdeg_G(v)\leq k'\leq\lceil k\rceil<\alpha$.
	If $\abs{\sigma_s^{-1}(v)}>1$, then by the construction of~$G_s$, $v=v_{i,1}$ for some $i\in[k']$, and therefore by Lemma~\ref{lem:contraction},
	\begin{linenomath*}\postdisplaypenalty=0\begin{align*}
		\rdeg_{G_s}(v)
		\leq\deg_{G_s}(v)
		\leq\sum_{w\in\sigma_s^{-1}(v)}\deg_G(w)
		\leq\sum_{w\in B_i}\deg_G(w)
		\leq\frac{dn}{k}.
	\end{align*}\end{linenomath*}
	Since
	\begin{linenomath*}\postdisplaypenalty=0\begin{align*}
		k\alpha
		&=\frac{3}{2}\left(q-\frac{\sqrt{q\ln q}+\sqrt{q}}{6}-\frac{\ln q}{9}\right)\left(q+\frac{\sqrt{q\ln q}+\sqrt{q}}{6}+\frac{5\ln q}{9}\right)\\
		&=\frac{3}{2}\left(\left(q+\frac{2\ln q}{9}\right)^2-\left(\frac{\sqrt{q\ln q}+\sqrt{q}}{6}+\frac{\ln q}{3}\right)^2\right)\\
		&>\frac{3}{2}\left(\left(q+\frac{2\ln q}{9}\right)^2-\frac{4q\ln q}{9}\right)\\
		&=\frac{3}{2}\left(q^2+\frac{4\ln^2q}{81}\right)=dn+\frac{2\ln^2q}{27}>dn,
	\end{align*}\end{linenomath*}
	we have $dn/k<\alpha$, so that $\rdeg_{G_s}(v)<\alpha$.
	Therefore, $\DeltaR(G_s)<\alpha$ for every $s\leq t$.
	
	Secondly, suppose that $t<s<t+r$.
	By the construction of $G_s$, if $v=v_{i,1}$ for some $i\in[k']$, then $\sigma_s^{-1}(v)=\sigma_t^{-1}(v)=B_i$.
	Therefore, by Lemma~\ref{lem:contraction},
	\begin{linenomath*}\[
		\rdeg_{G_s}(v)\leq\deg_{G_s}(v)\leq\sum_{w\in \sigma^{-1}_s(v)}\deg_G(w)=\sum_{w\in B_i}\deg_G(w)\leq\frac{dn}{k}<\alpha.
	\]\end{linenomath*}
	If $v\in V(G_s)\setminus\{v_{i,1}:i\in[k']\}$, then since $G_s\setminus\{v_{i,1}:i\in[k']\}=H_{s-t+1}$,
	\begin{linenomath*}\postdisplaypenalty=0\begin{align*}
		\rdeg_{G_s}(v)
		&\leq\abs{\{v_{i,1}:i\in[k']\}}+\rdeg_{H_{s-t+1}}(v)\\
		&<k+1+\frac{r+\sqrt{r\ln r}+\sqrt{r}+2\ln r}{2}\\
		&<\frac{3k}{2}+\frac{\sqrt{q\ln q}+\sqrt{q}+2\ln q}{2}+1\\
		&=\frac{3}{2}\left(q-\frac{\sqrt{q\ln q}+\sqrt{q}}{6}-\frac{\ln q}{9}\right)+\frac{\sqrt{q\ln q}+\sqrt{q}+2\ln q}{2}+1\\
		&=\alpha+1.
	\end{align*}\end{linenomath*}
	Since $\rdeg_{G_s}(v)$ is an integer, we have $\rdeg_{G_s}(v)\leq\lceil\alpha\rceil$.
	Therefore, $\DeltaR(G_s)\leq\lceil\alpha\rceil$ for every $s$ with $t<s<t+r$.

	Thirdly, if $s\geq t+r$, then
	\begin{linenomath*}\[
		\DeltaR(G_s)<\abs{V(G_s)}=n-s+1\leq n-t-r+1 = k'\leq\lceil k\rceil<\alpha,
	\]\end{linenomath*}
	and therefore $\tww(G)\leq\lceil\alpha\rceil$.
	By the mean value theorem,
	\begin{linenomath*}\[
		\frac{\sqrt{\ln(4m/3)} - \sqrt{\ln m}}{m/3} = \frac{1}{2 m' \sqrt{\ln m'}}
	\]\end{linenomath*}
	for some $m' \in (m, 4m/3)$.
	Hence, $\sqrt{\ln(4m/3)} - \sqrt{\ln m} < 1/(6\sqrt{\ln m})$.
	Then
	\begin{linenomath*}\postdisplaypenalty=0\begin{align*}
		&\alpha - \sqrt{3m}+1\\
		&=\frac{\sqrt{q\ln q}}{4}+\frac{\sqrt{q}}{4}+\frac{5\ln q}{6}+1\\
		&=\frac{\sqrt{\sqrt{m/3}\cdot \ln (4m/3) }}{4}
		+\frac{\sqrt{2} m^{1/4}}{4 \cdot 3^{1/4}}
		+\frac{5(\ln m+\ln(4/3))}{12}+1\\
		&<\frac{1}{4 \cdot 3^{1/4}} \left( m^{1/4} \sqrt{\ln m} + \frac{m^{1/4}}{6\sqrt{\ln m}}  \right)
		+\frac{\sqrt{2} m^{1/4}}{4 \cdot 3^{1/4}}
		+\frac{5(\ln m+\ln(4/3))}{12}+1\\
		&<\frac{m^{1/4} \sqrt{\ln m}}{4 \cdot 3^{1/4}} + \frac{m^{1/4}}{24 \cdot 3^{1/4} \sqrt{\ln m}} +\frac{\sqrt{2} m^{1/4}}{4 \cdot 3^{1/4}} + 
		\frac{5 \ln m}{12} + \frac{9}{8}\\
		&=\frac{m^{1/4} \sqrt{\ln m}}{4 \cdot 3^{1/4}} + m^{1/4} \left(\frac{1}{24 \cdot 3^{1/4} \sqrt{\ln m}}+\frac{\sqrt{2}}{4 \cdot 3^{1/4}}+\frac{5\ln m}{12 m^{1/4}} + \frac{9}{8 m^{1/4}}\right)\\
		&<\frac{m^{1/4} \sqrt{\ln m}}{4 \cdot 3^{1/4}} +\frac{3 m^{1/4}}{2},
	\end{align*}\end{linenomath*}
	where the last inequality holds because
	\begin{linenomath*}\[
		\frac{1}{24 \cdot 3^{1/4} \sqrt{\ln x}}+\frac{\sqrt{2}}{4 \cdot 3^{1/4}}+\frac{5\ln x}{12 x^{1/4}} + \frac{9}{8 x^{1/4}}
	\]\end{linenomath*} 
	is decreasing if $x\geq4$.
	Thus,
	\begin{linenomath*}\[
		\tww(G)<\alpha+1< \sqrt{3m}+ \frac{m^{1/4} \sqrt{\ln m}}{4\cdot 3^{1/4}} + \frac{3m^{1/4}}{2}.\qedhere
	\]\end{linenomath*}
\end{proof}

\section{Twin-width and random graphs}\label{sec:asymptotic}

We now prove Theorems~\ref{main3} and~\ref{smallp}.
First, we present the proof of Theorem~\ref{main3}, which provides a lower bound for the twin-width of $G(n,p)$.
Since a graph and its complement have the same twin-width~\cite{twin-width1}, we focus on random graphs $G(n,p)$ with $0<p\leq1/2$.

\thirdmain*

\begin{proof}
	By Lemma~\ref{lem:firststep}, it suffices to show that for all distinct $i,j\in[n]$,
	\begin{linenomath*}\[
		\abs{(N_G(i)\triangle N_G(j))\setminus\{i,j\}} > 2 p(1-p) n - (2\sqrt{2}+\varepsilon)(p(1-p)n \ln n)^{1/2}
	\]\end{linenomath*}
	asymptotically almost surely.
	Let $s := 1 - p^2 - (1-p)^2 = 2p(1-p)$.
	For all distinct $i,j\in[n]$ and each $k\in[n]\setminus\{i,j\}$, let $X^{\{i,j\}}_k$ be the random variable such that
	\begin{linenomath*}\[
		X^{\{i,j\}}_k = \begin{cases}
			1 & \text{if exactly one of $i$ or $j$ is adjacent to $k$ in $G$,}\\
			0 & \text{otherwise},
		\end{cases}
	\]\end{linenomath*}
	which indicates whether $k$ is incident with a red edge in $G/\{i,j\}$.
	For all distinct $i,j\in[n]$, let $X^{\{i,j\}}:=\sum_{\ell\in[n]\setminus\{i,j\}}X^{\{i,j\}}_\ell$.
	It is readily seen that $\E[X^{\{i,j\}}]= s(n-2)$ for all distinct $i,j\in[n]$.
	Let $C:=2+\varepsilon/\sqrt{2}$ and $D:=2+\varepsilon/2$.
	For all distinct $i,j\in[n]$, let $A^{\{i,j\}}$ be the event that $X^{\{i,j\}} > sn - C(sn\ln n)^{1/2}$.
	Then it suffices to show that $\Pr[ \bigcap_{1\leq i<j\leq n}A^{\{i,j\}}]\rightarrow1$ as $n\rightarrow\infty$.
	Since $1/n \leq p \leq 1/2$, we have $sn\geq 1$.
	Thus, for all sufficiently large~$n$,
	\begin{linenomath*}\postdisplaypenalty=0\begin{align*}
		&C(sn\ln n)^{1/2} - D(s(n-2)\ln n)^{1/2}
		> (C-D)(sn\ln n)^{1/2}
		\geq 2 \geq 2s.
	\end{align*}\end{linenomath*}
	That is,
	\begin{linenomath*}\postdisplaypenalty=0\begin{align*}
		-C(sn\ln n)^{1/2}<
		-2s- D(s(n-2)\ln n)^{1/2}.
	\end{align*}\end{linenomath*}
	Then by the union bound and Chernoff bound,
	\begin{linenomath*}\postdisplaypenalty=0\begin{align*}
		\Pr\left[ \bigcup_{1\leq i<j\leq n}\overline{A^{\{i,j\}}} \right]
		&\leq\binom{n}{2}\Pr\left[ \overline{A^{\{1,2\}}} \right]\\
		&=\binom{n}{2}\Pr[X^{\{1,2\}} \leq sn - C(sn\ln n)^{1/2}]\\
		&\leq\binom{n}{2}\Pr[X^{\{1,2\}} \leq s(n-2) - D(s(n-2)\ln n)^{1/2}]\\
		&\leq\binom{n}{2} \exp\left(-\frac{D^2s(n-2)\ln n}{2s(n-2)}\right)\\
		&=\binom{n}{2} n^{-D^2/2} <n^{2-D^2/2}\rightarrow0
	\end{align*}\end{linenomath*}
	as $n\rightarrow\infty$, and this completes the proof.
\end{proof}

As a corollary, we remark that if $(2+\delta)n^{-1}\ln n\leq p\leq 1/2$ for some positive real $\delta$, then $\tww(G(n,p)) = \Omega(\ln n)$ asymptotically almost surely.

We now prove Theorem~\ref{smallp}.
We will use the following three theorems of Erd\H{o}s and R\'{e}nyi~\cite{Erdoes1960}.

\begin{THM}[Erd\H{o}s and R\'{e}nyi~\cite{Erdoes1960}]\label{thm:unicyclic}
	Let $c$ be a constant in $(0,1)$.
	The random graph $G:=G(n,c/n)$ asymptotically almost surely has the property that every component of $G$ has at most one cycle.
\end{THM}

A graph $G$ is \emph{balanced} if for every non-null subgraph $H$ of $G$, the average degree of $H$ is at most the average degree of $G$.
We remark that all trees and cycles are balanced.

\begin{THM}[Erd\H{o}s and R\'{e}nyi~\cite{Erdoes1960}]\label{thm:balanced}
	Let $p:\mathbb{N}\rightarrow[0,1]$ be a function and~$H$ be a balanced graph with $v$ vertices and $e$ edges.
	If $p = o(n^{-v/e})$, then $G(n,p)$ has no subgraph isomorphic to $H$ asymptotically almost surely.
	If $p = \omega(n^{-v/e})$, then $G(n,p)$ has a subgraph isomorphic to $H$ asymptotically almost surely.
\end{THM}

\begin{THM}[Erd\H{o}s and R\'{e}nyi~\cite{Erdoes1960}]\label{thm:cycle}
	For a function $p : \mathbb{N} \to [0,1]$ with $p = o(n^{-1})$, $G(n,p)$ is a forest asymptotically almost surely.
\end{THM}

We also show that if a connected graph has at most one cycle, then its twin-width is at most $2$.
We will use the following $2$-contraction sequence of a tree~\cite{twin-width1}.

\begin{PROP}[Bonnet et al.~\cite{twin-width1}]\label{prop:tree}
	Every rooted tree admits a $2$-contraction sequence such that the root is contracted only in the last contraction.
\end{PROP}

\begin{PROP}\label{prop:unicyclic}
	If every component of a graph $G$ has at most one cycle, then $\tww(G)\leq2$.
\end{PROP}
\begin{proof}
	Since the twin-width of a graph is the maximum twin-width of its components~\cite{twin-width1}, we may assume that $G$ is connected.
	By Proposition~\ref{prop:tree}, we may assume that $G$ has exactly one cycle $C$.
	We are going to construct a $2$-contraction sequence of $G$.
	Let $v_1,\ldots,v_\ell$ be the vertices of $C$ such that $v_iv_j\in E(C)$ if either $\abs{i-j}=1$ or $\{i,j\}=\{1,\ell\}$.
	For each $i\in[\ell]$, let $T_i$ be the component of $G\setminus\{v_{i-1},v_{i+1}\}$ containing $v_i$ where $v_0:=v_n$ and $v_{\ell+1}:=v_1$ which is a tree rooted at $v_i$.
	For each $i\in[\ell]$, by Proposition~\ref{prop:tree}, there exists a partial $2$-contraction sequence $\sigma_i$ from $T_i$ to some $T'_i$ such that $\sigma_i^{-1}(v_i)=\{v_i\}$ and $T'_i$ has at most two vertices.
	By applying $\sigma_1,\ldots,\sigma_\ell$ in sequence, we obtain a partial contraction sequence $\sigma$ from $G$ to some trigraph $G'$.
	Note that for each $i\in[\ell]$, $v_i$ is not contracted during the process of constructing $G'$, so no edge of $C$ is a red edge of $G'$ and $\sigma$ is a partial $2$-contraction sequence.
	
	Now, we present a $2$-contraction sequence of $G'$.	
	Since taking induced subgraph does not increase twin-width, we may assume that for each $i\in[\ell]$, $T'_i$ has two vertices, say $v_i$ and $w_i$.
	For convenience, if we contract vertices $v_i$ and $w_j$ for some $i,j\in[\ell]$, then we call the new vertex $v_i$.
	Let $G_1:=G'/\{v_1,w_1\}$.
	For each $i\in[\ell]\setminus\{1\}$, let $G_i:=G_{i-1}/\{v_{i-1},w_i\}$.
	One may observe that for each $i\in[\ell]$, $\DeltaR(G_i)=2$, and $G_\ell$ is a cycle.
	Thus, by contracting adjacent vertices of $G_\ell$ iteratively, we find a $2$-contraction sequence of $G'$.
\end{proof}

\begin{figure}
	\centering
	\tikzstyle{v}=[circle, draw, solid, fill=black, inner sep=0pt, minimum width=3pt]
	\begin{tikzpicture}[scale=0.75]
		\draw (0,1) node[v](t){};
		\foreach \x in {0,1,2} {
			\draw (-1+\x,0) node[v](x\x){};
			\draw (-1+\x,-1) node[v](y\x){};
			\draw (t)--(x\x)--(y\x);
		}
	\end{tikzpicture}
	\caption{A $1$-subdivision of $K_{1,3}$.}
	\label{fig:1}
\end{figure}
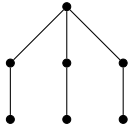
		
A \emph{caterpillar} is a tree that contains a path $P$ such that every vertex outside $P$ is a leaf whose neighbor is in $V(P)$.
It is easy to observe that a tree is a caterpillar if and only if it has no $1$-subdivision of $K_{1,3}$ (see Figure~\ref{fig:1}) as a subgraph.

\begin{LEM}\label{lem:cater}
	For a tree $T$, $\tww(T)\leq1$ if and only if $T$ is a caterpillar.
\end{LEM}
\begin{proof}
	We first show that $\tww(T)\leq 1$ if $T$ is a caterpillar.
	We may assume that no pair of leaves are adjacent to the same vertex. 
	Hence, for some integer $n$, $T$ is isomorphic to an induced subgraph of the tree obtained from the path $v_2v_4\cdots v_{2n}$ by appending a leaf $v_{2i-1}$ to the vertex $v_{2i}$ for each $i\in [n]$. 
	We find a $1$-contraction sequence for this graph by iteratively contracting $v_i$ and $v_{i+1}$ into $v_{i+1}$ for $i\in [2n-1]$.

	For the forward direction, it suffices to show that the $1$-subdivision of $K_{1,3}$ has twin-width $2$.
	Let $H$ be the $1$-subdivision of $K_{1,3}$.
	It is readily seen that for distinct vertices $v$ and $w$ of~$H$, $\abs{(N_H(v)\triangle N_H(w))\setminus\{v,w\}}\leq1$ if and only if $v$ and $w$ are adjacent and one of them is a leaf.
	Thus, by Lemma~\ref{lem:firststep}, we may assume that we first contract such vertices, say $v$ and $w$.
	Then $H/\{v,w\}$ has no two vertices whose contraction yields a $1$-trigraph.
\end{proof}

We now prove Theorem~\ref{smallp}.

\smallp*

\begin{proof}
	If $p=o(n^{-4/3})$, then by Theorems~\ref{thm:balanced} and~\ref{thm:cycle}, $G$ has no induced path of length $3$ asymptotically almost surely.
	A graph has twin-width $0$ if and only if it has no induced path of length $3$~\cite{twin-width1}, and therefore $\tww(G)=0$ asymptotically almost surely.

	If $p = \omega(n^{-4/3})$ and $p = o(n^{-7/6})$, then by Theorems~\ref{thm:balanced} and~\ref{thm:cycle}, $G$ is a forest and contains a path of length $3$ and no $1$-subdivision of $K_{1,3}$ as a subgraph asymptotically almost surely.
	Then by Lemma~\ref{lem:cater}, $\tww(G) = 1$ asymptotically almost surely.

	If $p = \omega(n^{-7/6})$ and $p = o(n^{-1})$, then by Theorems~\ref{thm:balanced} and~\ref{thm:cycle}, $G$ is a forest and contains the $1$-subdivision of $K_{1,3}$ as a subgraph asymptotically almost surely.
	Then by Proposition~\ref{prop:tree} and Lemma~\ref{lem:cater}, $\tww(G) = 2$ asymptotically almost surely.
	
	Now, suppose that $p = \omega(n^{-9/8})$ and $p\leq c/n$ for a constant $c \in (0,1)$.
	By Theorem~\ref{thm:unicyclic}, every component of $G$ has at most one cycle asymptotically almost surely, and therefore $\tww(G) \leq 2$ asymptotically almost surely by Proposition~\ref{prop:unicyclic}.
	By Theorem~\ref{thm:balanced}, $G$ has the $1$-subdivision of $K_{1,4}$ as a subgraph asymptotically almost surely, and therefore has the $1$-subdivision of $K_{1,3}$ as an induced subgraph.
	Since the $1$-subdivision of $K_{1,3}$ has twin-width $2$, we have $\tww(G) \geq 2$ asymptotically almost surely.
\end{proof}

\providecommand{\bysame}{\leavevmode\hbox to3em{\hrulefill}\thinspace}
\providecommand{\MR}{\relax\ifhmode\unskip\space\fi MR }
\providecommand{\MRhref}[2]{%
  \href{http://www.ams.org/mathscinet-getitem?mr=#1}{#2}
}
\providecommand{\href}[2]{#2}


\begin{thebibliography}{10}

\bibitem{BH2021}
Jakub Balab\'{a}n and Petr Hlin\v{e}n\'{y}, \emph{{Twin-Width Is Linear in the
  Poset Width}}, 16th International Symposium on Parameterized and Exact
  Computation (IPEC 2021) (Dagstuhl, Germany) (Petr~A. Golovach and Meirav
  Zehavi, eds.), Leibniz International Proceedings in Informatics (LIPIcs),
  vol. 214, Schloss Dagstuhl -- Leibniz-Zentrum f{\"u}r Informatik, 2021,
  \href{https://drops.dagstuhl.de/opus/volltexte/2021/15389}{\path{doi:
  10.4230/LIPIcs.IPEC.2021.6}}, pp.~6:1--6:13.

\bibitem{Bona2015}
Mikl\'{o}s B\'{o}na (ed.), \emph{Handbook of enumerative combinatorics},
  Discrete Mathematics and its Applications (Boca Raton), CRC Press, Boca
  Raton, FL, 2015, \href{https://doi.org/10.1201/b18255}{\path{doi:
  10.1201/b18255}}. \MR{3408702}

\bibitem{gridtwinwidth}
\'{E}douard Bonnet, \emph{{T}win-width {I}: tractable {FO} model checking},
  {Y}ou{T}ube, Nov. 3, 2020, \url{https://www.youtube.com/watch?v=DFMEiVz_A3Q}.

\bibitem{twin-width3}
\'{E}douard Bonnet, Colin Geniet, Eun~Jung Kim, St\'{e}phan Thomass\'{e}, and
  R\'{e}mi Watrigant, \emph{Twin-width {III}: max independent set, min
  dominating set, and coloring}, 48th {I}nternational {C}olloquium on
  {A}utomata, {L}anguages, and {P}rogramming, LIPIcs. Leibniz Int. Proc.
  Inform., vol. 198, Schloss Dagstuhl. Leibniz-Zent. Inform., Wadern, 2021,
  \href{https://drops.dagstuhl.de/opus/volltexte/2021/14104/}{\path{doi:
  10.4230/LIPIcs.ICALP.2021.35}}, pp.~Art. No. 35, 20. \MR{4288865}

\bibitem{twin-width2}
\bysame, \emph{Twin-width {II}: small classes}, Combinatorial Theory
  \textbf{2(2)} (2022), \href{https://doi.org/10.5070/C62257876}{\path{doi:
  10.5070/C62257876}}.

\bibitem{twin-width4}
\'{E}douard Bonnet, Ugo Giocanti, Patrice Ossona~de Mendez, Pierre Simon,
  St\'{e}phan Thomass\'{e}, and Szymon Toru\'{n}czyk, \emph{Twin-width {IV}:
  Ordered graphs and matrices}, Proceedings of the 54th Annual ACM SIGACT
  Symposium on Theory of Computing (New York, NY, USA), STOC 2022, Association
  for Computing Machinery, 2022,
  \href{https://doi.org/10.1145/3519935.3520037}{\path{doi:
  10.1145/3519935.3520037}}, p.~924–937.

\bibitem{twin-width6}
\'{E}douard Bonnet, Eun~Jung Kim, Amadeus Reinald, and St\'{e}phan
  Thomass\'{e}, \emph{Twin-width {VI}: the lens of contraction sequences},
  Proceedings of the 2022 {A}nnual {ACM}-{SIAM} {S}ymposium on {D}iscrete
  {A}lgorithms ({SODA}), [Society for Industrial and Applied Mathematics
  (SIAM)], Philadelphia, PA, 2022,
  \href{https://doi.org/10.1137/1.9781611977073.45}{\path{doi:
  10.1137/1.9781611977073.45}}, pp.~1036--1056. \MR{4415082}

\bibitem{BKRTW2021}
{\'E}douard Bonnet, Eun~Jung Kim, Amadeus Reinald, St{\'e}phan Thomass{\'e},
  and R{\'e}mi Watrigant, \emph{Twin-width and polynomial kernels},
  Algorithmica (2022),
  \href{https://doi.org/10.1007/s00453-022-00965-5}{\path{doi:
  10.1007/s00453-022-00965-5}}.

\bibitem{twin-width1}
\'{E}douard Bonnet, Eun~Jung Kim, St\'{e}phan Thomass\'{e}, and R\'{e}mi
  Watrigant, \emph{Twin-width {I}: {T}ractable {FO} model checking}, J. ACM
  \textbf{69} (2022), no.~1, Art. 3, 46,
  \href{https://doi.org/10.1145/3486655}{\path{doi: 10.1145/3486655}}.
  \MR{4402362}

\bibitem{BNOST2021}
{\'E}douard Bonnet, Jaroslav Nešetřil, Patrice Ossona~de Mendez, Sebastian
  Siebertz, and St\'{e}phan Thomass\'{e}, \emph{Twin-width and permutations},
  \href{https://arxiv.org/abs/2102.06880}{\path{arXiv:2102.06880}}, 2021.

\bibitem{DGJOR2021}
Jan Dreier, Jakub Gajarsk\'{y}, Yiting Jiang, Patrice Ossona~de Mendez, and
  Jean-Florent Raymond, \emph{Twin-width and generalized coloring numbers},
  Discrete Math. \textbf{345} (2022), no.~3, Paper No. 112746, 8,
  \href{https://doi.org/10.1016/j.disc.2021.112746}{\path{doi:
  10.1016/j.disc.2021.112746}}. \MR{4349879}

\bibitem{Erdoes1960}
Paul Erd\H{o}s and Alfr\'{e}d R\'{e}nyi, \emph{On the evolution of random
  graphs}, A Magyar Tudom\'{a}nyos Akad\'{e}mia. Matematikai Kutat\'{o}
  Int\'{e}zet\'{e}nek K\"{o}zlem\'{e}nyei \textbf{5} (1960), 17--61.
  \MR{125031}

\bibitem{ER1963}
\bysame, \emph{Asymmetric graphs}, Acta Math. Acad. Sci. Hungar. \textbf{14}
  (1963), 295--315, \href{https://doi.org/10.1007/BF01895716}{\path{doi:
  10.1007/BF01895716}}. \MR{156334}

\bibitem{GPT2021}
Jakub Gajarský, Michał Pilipczuk, and Szymon Toruńczyk, \emph{Stable graphs
  of bounded twin-width},
  \href{https://arxiv.org/abs/2107.03711}{\path{arXiv:2107.03711}}, 2021.

\bibitem{randomgraphs}
Svante Janson, Tomasz {\L}uczak, and Andrzej Rucinski, \emph{Random graphs},
  Wiley-Interscience Series in Discrete Mathematics and Optimization,
  Wiley-Interscience, New York, 2000,
  \href{https://onlinelibrary.wiley.com/doi/book/10.1002/9781118032718}{\path{doi:
  10.1002/9781118032718}}. \MR{1782847}

\bibitem{squaregrid}
V\'{\i}t Jel\'{\i}nek, \emph{The rank-width of the square grid}, Discrete Appl.
  Math. \textbf{158} (2010), no.~7, 841--850,
  \href{https://doi.org/10.1016/j.dam.2009.02.007}{\path{doi:
  10.1016/j.dam.2009.02.007}}. \MR{2602833}

\bibitem{Joh98}
\"{O}jvind Johansson, \emph{{G}raph decomposition using node labels}, Ph.D.
  thesis, Royal Institute of Technology, 2001.

\bibitem{pathbound}
Joachim Kneis, Daniel M\"{o}lle, Stefan Richter, and Peter Rossmanith, \emph{A
  bound on the pathwidth of sparse graphs with applications to exact
  algorithms}, SIAM J. Discrete Math. \textbf{23} (2008/09), no.~1, 407--427,
  \href{https://doi.org/10.1137/080715482}{\path{ doi: 10.1137/080715482}}.
  \MR{2476839}

\bibitem{LLO2012}
Choongbum Lee, Joonkyung Lee, and Sang-il Oum, \emph{Rank-width of random
  graphs}, J. Graph Theory \textbf{70} (2012), no.~3, 339--347,
  \href{https://doi.org/10.1002/jgt.20620}{\path{doi: 10.1002/jgt.20620}}.
  \MR{2946080}

\bibitem{SS2021}
Andr{\'e} Schidler and Stefan Szeider, \emph{A {SAT} approach to twin-width},
  2022 Proceedings of the Symposium on Algorithm Engineering and Experiments
  (ALENEX), SIAM, 2022,
  \href{https://doi.org/10.1137/1.9781611977042.6}{\path{doi:
  10.1137/1.9781611977042.6}}, pp.~67--77.

\bibitem{ST2021}
Pierre Simon and Szymon Toruńczyk, \emph{Ordered graphs of bounded
  twin-width},
  \href{https://arxiv.org/abs/2102.06881}{\path{arXiv:2102.06881}}, 2021.

\end{thebibliography}
\end{document}